\documentclass[letterpaper, 10 pt, conference]{ieeeconf}  

\IEEEoverridecommandlockouts                             
\usepackage{float}
\usepackage{verbatim}
\usepackage{paralist}
\usepackage{amsfonts}
\usepackage{graphics} 
\usepackage{epsfig} 
\usepackage{mathptmx} 
\usepackage{times} 
\usepackage{amsmath} 
\usepackage{amssymb}  
\usepackage{physics}
\usepackage{mdwmath}
\usepackage{mdwtab}
\usepackage{color}
\usepackage[hidelinks]{hyperref}
\usepackage{algorithm}
\usepackage{algorithmic}

\usepackage{caption}
\usepackage{subcaption}
\usepackage{amsmath} 
\usepackage{bm}      

\usepackage{tikz}
\usetikzlibrary{shapes, arrows.meta, positioning}

\newtheorem{proposition}{Proposition}%

\newtheorem{remark}{Remark}

\def\BibTeX{{\rm B\kern-.05em{\sc i\kern-.025em b}\kern-.08em
    T\kern-.1667em\lower.7ex\hbox{E}\kern-.125emX}}
\begin{document}

\title{Quantum-Inspired Tensor Networks for Approximating PDE Flow Maps}

\author{Nahid Binandeh Dehaghani, Ban Q. Tran, Rafal Wisniewski, Susan Mengel, and A. Pedro Aguiar
\thanks{Nahid Binandeh Dehaghani is with the Department of Electronic Systems, Aalborg University, Aalborg, Denmark.
        {\tt\small nahidbd@es.aau.dk}}%
\thanks{Ban Q. Tran is with the Department of Computer Science, Texas Tech University, Lubbock, USA, and Department of Computing Fundamentals, FPT University, Hanoi, Vietnam.
        {\tt\small bantran@ttu.edu - bantq3@fe.edu.vn}}
\thanks{Rafal Wisniewski is with the Department of Electronic Systems, Aalborg University, Aalborg, Denmark. {\tt\small raf@es.aau.dk}}
\thanks{Susan Mengel is with the Department of Computer Science, Texas Tech University, Lubbock, USA. {\tt\small susan.mengel@ttu.edu}} 
\thanks{A. Pedro Aguiar is with SYSTEC-ARISE, Faculty of Engineering, University of Porto, Porto, Portugal. {\tt\small 	pedro.aguiar@fe.up.pt}}
 }

\maketitle

\begin{abstract}
We investigate quantum-inspired tensor networks (QTNs) for approximating flow maps of hydrodynamic partial differential equations (PDEs). Motivated by the effective low-rank structure that emerges after tensorization of discretized transport and diffusion dynamics, we encode PDE states as matrix product states (MPS) and represent the evolution operator as a structured low-rank matrix product operator (MPO) in tensor-train form (e.g., arising from finite-difference discretizations assembled in MPO form).
The MPO is applied directly in MPS form, and rank growth is controlled via canonicalization and SVD-based truncation after each step.
We provide theoretical context through standard matrix product properties, including exact MPS representability bounds, local optimality of SVD truncation, and a Lipschitz-type multi-step error propagation estimate. Experiments on one- and two-dimensional linear advection--diffusion and nonlinear viscous Burgers equations demonstrate accurate short-horizon prediction, favorable scaling in smooth diffusive regimes, and error growth in nonlinear multi-step predictions.
\end{abstract}

\section{Introduction}
\label{sec:introduction}

Hydrodynamic partial differential equations (PDEs) governing transport and diffusion
arise across fluid dynamics, materials transport, and nonlinear wave propagation.
Although structurally simple, such equations can develop sharp gradients,
multiscale interactions, and sensitivity to perturbations, making accurate and scalable
numerical simulation increasingly challenging as spatial resolution and dimension increase.

An alternative to classical discretization-based solvers is to approximate the discrete-time
\emph{flow map} that advances the state over a fixed time step. This viewpoint is closely related to
Koopman/operator perspectives
\cite{Koopman1931, Brunton2016}, where dynamics are represented through linear operators acting on
high-dimensional state spaces. However, for spatially discretized PDEs the state dimension grows
exponentially with resolution, rendering naive operator representations and applications computationally infeasible.

Tensor network representations provide a principled mechanism for mitigating this
curse of dimensionality. Originally developed in quantum many-body physics
\cite{Schollwock2011, Orus2014}, tensor networks exploit low-rank structure to represent
high-dimensional states and operators efficiently. In particular, matrix product
states (MPS), also known as tensor trains \cite{Orus2014, Oseledets2011}, admit representations
whose complexity scales linearly with dimension when correlations remain sufficiently
localized. Such quantum-inspired representations have increasingly been adopted in
classical scientific computing for compression and structured operator approximation
\cite{Cichocki2016,Peddinti2024CFD,Gourianov2022Turbulence,Kiffner2023ROM}. Throughout, \emph{QTN} refers to the quantum-inspired tensor-network framework instantiated with MPS (tensor trains) for states and 
matrix product operator (MPO) for operators.

Many transport-dominated PDEs exhibit effective low-rank structure over finite
time horizons: smooth initial data, diffusive effects, and locality of interactions
tend to limit long-range correlations in discretized states \cite{Peddinti2024CFD,Gourianov2022Turbulence,Kiffner2023ROM}. This suggests that
PDE flow maps may admit compact tensor-network representations. At the same time,
nonlinear dynamics can generate complex correlations that challenge low-rank
approximations, motivating a systematic study of both the capabilities and limitations
of tensor-network time stepping.
We make the following contributions:
\begin{itemize}
\item We investigate a quantum-inspired tensor network framework
for approximating PDE flow maps, representing states as MPS and
evolution operators as low-rank MPOs in tensor-train form.
\item We provide theoretical context grounded in standard matrix product
properties, including exact MPS representability bounds, optimality of SVD-based truncation, and a Lipschitz-type multi-step error propagation estimate.
\item We perform systematic numerical experiments on linear and nonlinear
hydrodynamic PDEs in one and two spatial dimensions, benchmarking short-horizon accuracy, scalability, and long-horizon stability against reference solvers. We empirically characterize the limitations of low-rank approximations with truncation in nonlinear regimes, establishing QTNs as a principled
baseline for future physics-informed and data-driven extensions.
\end{itemize}

\paragraph*{Organization of the Paper} The remainder of this paper is organized as follows. Section II reviews the tensor-network preliminaries, including MPS and MPO representations. Section III introduces the proposed QTN framework for tensorized PDE state representation and low-rank one-step operator modeling. Section IV presents the compressed time-stepping methodology and key theoretical properties. Section V reports numerical experiments on linear advection–diffusion and nonlinear viscous Burgers equations in one and two dimensions. Section VI concludes the paper and outlines future directions.

\section{Preliminaries}
\label{sec:preliminaries}

This section introduces the tensor-network concepts and notation used in this
work, focusing on matrix product representations of high-dimensional states
(MPS/TT) and operators (MPO/TT-operators).

\subsection{Matrix Product States}

Matrix Product States (MPS), also known as tensor trains, provide structured
low-rank representations of high-dimensional tensors and originate from
quantum many-body physics. In our setting, they serve as a compression
framework for tensorized discretized PDE solution fields and for efficient
application of structured operators. We summarize the notation and canonical
forms used throughout the paper.

\subsubsection{High-Dimensional States and Tensorization}

Let a discretized PDE state sampled on $N$ spatial grid points be represented
as a vector $u \in \mathbb{R}^{N}$. To enable tensor-network compression, we
\emph{tensorize} $u$ by reshaping it into an order-$n$ tensor
$\psi \in \mathbb{R}^{d \times \cdots \times d}$ such that $N = d^n$.\footnote{
If $N$ is not exactly of the form $d^n$, one may use padding or a mixed-radix
tensorization; we restrict to $N=d^n$ for clarity.}
Using $s_i \in \{0,1,\dots,d-1\}$ induces the expansion
\begin{equation*}
|\Psi\rangle
\in \bigotimes_{i=1}^{n} \mathbb{R}^{d}
\;\cong\; \mathbb{R}^{d^n},
\qquad
|\Psi\rangle
= \sum_{s_1,\dots,s_n}
\psi_{s_1 \dots s_n}\, |s_1\rangle \otimes \cdots \otimes |s_n\rangle,
\end{equation*}
where $\{|s_1\rangle \otimes \cdots \otimes |s_n\rangle\}$ denotes the
tensor-product basis associated with the chosen index encoding of the spatial
grid. In the present PDE setting, the indices $s_i$ do not correspond to
physical subsystems; they arise solely from the tensorization map.

Although storing $\psi$ still requires $d^n = N$ entries, dense representations
of linear maps on this space require $d^{2n} = N^2$ coefficients, and naive
operator application and contractions scale accordingly. This motivates representing both states and operators in structured low-rank tensor formats.

\subsubsection{Matrix Product State Representation}

An MPS factorizes the coefficient tensor $\psi_{s_1 \dots s_n}$ into a product
of site-local cores:
\begin{equation}
|\Psi\rangle
= \sum_{s_1,\dots,s_n}
A_1^{(s_1)} A_2^{(s_2)} \cdots A_n^{(s_n)}
\, |s_1 \dots s_n\rangle,
\end{equation}
and equivalently,
$\psi_{s_1 \dots s_n} = A_1^{(s_1)} A_2^{(s_2)} \cdots A_n^{(s_n)}$,
where matrix multiplication over the virtual indices is implied. Each core
$A_i^{(s_i)} \in \mathbb{R}^{D_{i-1} \times D_i}$, and
$(D_0,\dots,D_n)$ denote the bond dimensions with $D_0 = D_n = 1$.

The total number of parameters scales as
$\sum_{i=1}^{n} D_{i-1}\, d\, D_i$.
Under a uniform bound $D_i \le \chi$, this is $O(n d \chi^2)$, i.e., linear in
$n$ and polynomial in the maximal bond dimension $\chi = \max_i D_i$. When the
bond dimensions remain bounded (or grow slowly) as $n$ increases, this yields
substantial compression compared to storing the full tensor of size $d^n$.

\subsubsection{Canonical Forms}

MPS representations are not unique due to gauge freedom: inserting
$X X^{-1}$ between adjacent cores leaves the represented tensor unchanged.
Canonical forms partially fix this gauge and play a central role in numerical
stability and efficient computation \cite{Schollwock2011}.

\paragraph{Left-canonical form}
Applying successive singular value decompositions (SVDs) from left to right
yields MPS cores satisfying
\begin{equation}
\sum_{s_i}
\left(A_i^{(s_i)}\right)^\top A_i^{(s_i)} = I_{D_i},
\end{equation}
where $I_D$ denotes the $D \times D$ identity matrix.
This enforces orthonormality on the right virtual index (after reshaping the
core into a $(D_{i-1} d)\times D_i$ matrix). Throughout this work, we restrict
attention to real-valued tensor networks, so transposition and Hermitian
conjugation coincide.

\paragraph{Right-canonical form}
Applying SVDs from right to left yields cores satisfying
\begin{equation}
\sum_{s_i}
A_i^{(s_i)} \left(A_i^{(s_i)}\right)^\top = I_{D_{i-1}},
\end{equation}
which enforces orthonormality on the left virtual index (after reshaping the
core into a $D_{i-1}\times(d D_i)$ matrix).

\paragraph{Mixed-canonical form}
For a given central site $k$, an MPS can be written in mixed-canonical form as
\begin{equation}
|\Psi\rangle
= \sum_{s_1,\dots,s_n}
A_1^{(s_1)} \cdots A_{k-1}^{(s_{k-1})}\,
\Gamma_k^{(s_k)} \, C_k \,
A_{k+1}^{(s_{k+1})} \cdots A_n^{(s_n)}
\, |s_1 \dots s_n\rangle,
\end{equation}
where the tensors at sites $1,\dots,k-1$ are left-canonical and those at
sites $k+1,\dots,n$ are right-canonical. The diagonal (nonnegative) matrix $C_k$ contains
the singular values (Schmidt coefficients) associated with the bipartition at
bond $(k,k+1)$, and $\Gamma_k^{(s_k)} \in \mathbb{R}^{D_{k-1} \times D_k}$ is
the center tensor. Depending on convention, $C_k$ may be absorbed into the
center tensor or placed on the adjacent bond; we use the displayed convention
for convenience. This form explicitly separates orthonormal left and right
blocks and is used to improve conditioning and control bond-dimension growth.

\subsection{Matrix Product Operators}

Matrix Product Operators (MPOs) \cite{Schollwock2011, Orus2014} provide a
tensor-network representation of linear maps acting on high-dimensional states
expressed in MPS form. A linear operator
$\mathcal{O} : \mathbb{R}^{d^n} \to \mathbb{R}^{d^n}$ acts on a state
$|\Psi\rangle = \sum_{s'_1,\dots,s'_n} \psi_{s'_1 \dots s'_n}
|s'_1 \dots s'_n\rangle$ according to
\begin{equation}
\mathcal{O} |\Psi\rangle
=
\sum_{s_1,\dots,s_n}
\sum_{s'_1,\dots,s'_n}
O_{s_1 \dots s_n}^{\, s'_1 \dots s'_n}\,
\psi_{s'_1 \dots s'_n}\,
|s_1 \dots s_n\rangle,
\end{equation}
where $O_{s_1 \dots s_n}^{\, s'_1 \dots s'_n}$ are the coefficients of
$\mathcal{O}$ in the tensor-product basis. We use the convention that $s'$
indexes the \emph{input} (column) multi-index and $s$ indexes the \emph{output}
(row) multi-index. The full operator tensor has $d^{2n}$ entries and is
therefore prohibitively large for large $n$.

An MPO factorizes the operator tensor into a sequence of site-local cores:
\begin{equation}
O^{s'_1 \dots s'_n}_{s_1 \dots s_n}
=
\sum_{\alpha_1,\dots,\alpha_{n-1}}
W^{(s_1,s'_1)}_{1,\; 1,\alpha_1}
W^{(s_2,s'_2)}_{2,\; \alpha_1,\alpha_2}
\cdots
W^{(s_n,s'_n)}_{n,\; \alpha_{n-1},1}
\end{equation}
where each MPO core
$W_i \in \mathbb{R}^{R_{i-1} \times d \times d \times R_i}$ with
$R_0 = R_n = 1$. For fixed $(s_i,s'_i)$, the slice
$W_i^{(s_i,s'_i)}$ is a matrix in $\mathbb{R}^{R_{i-1}\times R_i}$, and
$R = \max_i R_i$ denotes the MPO bond dimension. This factorization exploits
low-rank structure to provide compact representations of high-dimensional
linear operators.

\begin{remark}[Action of an MPO on an MPS]
Given an MPS with cores $A_i^{(s_i)} \in \mathbb{R}^{D_{i-1} \times D_i}$, the
application of an MPO produces a new MPS whose site tensors are obtained by
contracting the physical input index and combining the virtual indices:
\begin{equation}
\widetilde{A}_i^{(s_i)}
=
\sum_{s'_i}
W_i^{(s_i,s'_i)} \otimes A_i^{(s'_i)},
\end{equation}
where $\widetilde{A}_i^{(s_i)} \in
\mathbb{R}^{(R_{i-1} D_{i-1}) \times (R_i D_i)}$ and $\otimes$ denotes the Kronecker product over the virtual indices. The resulting bond dimensions
reflect the combined ranks of the MPO and MPS and are typically reduced using
canonicalization and SVD-based truncation to enforce prescribed bond-dimension
limits. In practice, these contractions are implemented using optimized tensor
contraction paths to ensure computational efficiency.
\end{remark}

\begin{remark}[Relevance to PDE Flow Maps]
In our framework, the discrete-time PDE update
$u^{k+1} = \Phi_{\Delta t}(u^k)$ is represented (after tensorization/encoding)
by an MPO acting on an MPS representation of the discretized state. For
nonlinear PDEs, this MPO-based one-step update should be interpreted as a
low-rank approximation of the one-step map on the states encountered in
the rollout. This representation enables compressed operator representation and
application in high-dimensional settings by exploiting locality and low-rank
structure that often arise in transport- and diffusion-dominated regimes.
\end{remark}

\section{Quantum-Inspired Tensor Network (QTN) Models}
\label{sec:QTN}

This section introduces the tensor-network representations underlying the
proposed framework. We describe how discretized PDE states are encoded as MPS,
and how discrete-time flow maps are modeled using structured low-rank operators
represented in matrix product operator (MPO) or tensor-train (TT) form.

\subsection{Discretized PDE State Representation}

Consider a hydrodynamic PDE defined on a spatial domain
$\Omega \subset \mathbb{R}^{p}$ with solution $u(x,t)$ discretized on a grid of
$N=d^n$ points. For clarity, we begin with the one-dimensional case. At time
$t_k$, the discretized state is
\begin{equation}
\mathbf{u}_k =
\big[u(x_0,t_k),\, u(x_1,t_k),\, \dots,\, u(x_{N-1},t_k)\big]^\top
\in \mathbb{R}^N .
\end{equation}
Each spatial index $i \in \{0,\dots,N-1\}$ is written in base-$d$ form as
$i \equiv (i_1,i_2,\dots,i_n)$ with $i_j \in \{0,1,\dots,d-1\}$, which induces a
tensorization of $\mathbf{u}_k$ into an order-$n$ tensor
\begin{equation}
\mathcal{U}_k(i_1,i_2,\dots,i_n) = \mathbf{u}_k[i]
\in \mathbb{R}^{d \times \cdots \times d}.
\end{equation}
(Binary encoding corresponds to the special case $d=2$.) We denote by
$\mathrm{encode}:\mathbb{R}^N \to \mathbb{R}^{d^n}$ the fixed reshaping map
$\mathbf{u}_k \mapsto \Psi_k$, where $\Psi_k$ is the vectorization of
$\mathcal{U}_k$ (equivalently, the coefficient tensor in the tensor-product
basis of Section~\ref{sec:preliminaries}). The inverse reshaping map is denoted
$\mathrm{decode}$, used only for visualization/comparison in experiments.

Following Section~\ref{sec:preliminaries}, the tensor $\mathcal{U}_k$ is
approximated by an MPS of the form
\begin{equation}
\mathcal{U}_k(i_1,\dots,i_n)
\approx
\sum_{\alpha_1,\dots,\alpha_{n-1}}
A_1^{(i_1)} A_2^{(i_2)} \cdots A_n^{(i_n)},
\label{eq:mps_state_rep}
\end{equation}
where the MPS cores satisfy $A_j^{(i_j)} \in \mathbb{R}^{D_{j-1} \times D_j}$
with $D_0 = D_n = 1$. Canonical forms (left-, right-, and mixed-canonical) are
used to improve numerical stability and to control bond-dimension growth.

This construction extends to higher spatial dimensions by mapping a
multi-index grid coordinate (e.g., $(\ell_1,\dots,\ell_p)$) to a single
base-$d$ index $i$ and then applying the same tensorization; in practice, we
choose the index ordering (tensorization architecture) to preserve spatial
locality as much as possible.

\subsection{Time-Stepping Operator in Tensorized Coordinates}

Let $\mathcal{F}_{\Delta t}$ denote the exact (or reference) discrete-time flow
map of the PDE,
\begin{equation}
\mathbf{u}_{k+1} = \mathcal{F}_{\Delta t}(\mathbf{u}_k),
\end{equation}
which advances the discretized state over a fixed time step $\Delta t$.
Rather than forming $\mathcal{F}_{\Delta t}$ explicitly as a dense operator on
$\mathbb{R}^{N}$, we represent and apply a structured low-rank one-step operator
in tensorized coordinates using a tensor-network operator $\mathcal{T}$.

Let $\Psi_k = \mathrm{encode}(\mathbf{u}_k) \in \mathbb{R}^{d^n}$ denote the
tensorized representation of the discretized state. The tensor-network
time-stepping model is
\begin{equation}
\Psi_{k+1}^{\mathrm{QTN}} = \mathcal{T}(\Psi_k),
\label{eq:qtn_onestep}
\end{equation}
where $\mathcal{T} : \mathbb{R}^{d^n} \to \mathbb{R}^{d^n}$ is a linear operator
represented as a structured low-rank MPO (equivalently, a TT-operator). The
predicted physical state is obtained by decoding:
\begin{equation}
\mathbf{u}_{k+1}^{\mathrm{QTN}}
=
\mathrm{decode}\!\left(\Psi_{k+1}^{\mathrm{QTN}}\right)
=
\mathrm{decode}\!\left(\mathcal{T}(\Psi_k)\right).
\end{equation}

In implementation, applying an MPO to an MPS typically increases intermediate
bond dimensions. To control rank growth, we apply canonicalization and
SVD-based truncation after each operator application. We denote this
(non-differentiable) compression/projection by
$\Pi_{\chi_{\max},\varepsilon_{\mathrm{SVD}}}(\cdot)$, where $\chi_{\max}$ is a bond-dimension
cap and $\varepsilon_{\mathrm{SVD}}$ is an SVD threshold. The practical one-step predictor is
therefore
\begin{equation}
\Psi_{k+1}^{\mathrm{QTN}}
=
\Pi_{\chi_{\max},\varepsilon_{\mathrm{SVD}}}\!\left(\mathcal{T}(\Psi_k)\right).
\label{eq:qtn_projected}
\end{equation}

Although the exact flow map $\mathcal{F}_{\Delta t}$ may be nonlinear, the
MPO-based update \eqref{eq:qtn_projected} provides a low-rank \emph{approximation}
of the one-step evolution in tensorized coordinates on the states encountered
along the rollout. From a numerical perspective, $\mathcal{T}$ can be obtained
from a discretization of the PDE (e.g., finite-difference operators assembled
in MPO form) and combined with a chosen time-integration scheme  (e.g., explicit Euler), while
$\Pi_{\chi_{\max},\varepsilon_{\mathrm{SVD}}}$ controls complexity by preventing uncontrolled
rank growth during repeated time stepping.

\section{Methods} \label{sec:methods}

This section details the tensor-network formulation used for approximating
discrete-time PDE flow maps via compressed time stepping. Discretized states and
operators are represented using MPS and MPO structures, enabling compressed
manipulation of high-dimensional dynamics. Tensorization, MPO--MPS application,
canonicalization, and bond-dimension truncation are performed within the
tensor-network framework. Truncation is treated as a separate
(non-differentiable) projection step that controls complexity during rollout.

\subsection{Tensor-Network Time-Stepping Pipeline}
Algorithm~\ref{alg:qtn} summarizes the compressed QTN time-stepping loop. Each
step applies the MPO update to the current MPS state and truncates the resulting
bond dimensions to enforce the prescribed complexity limits.

\paragraph{Tensorization and State Encoding}
A discretized PDE state $\mathbf{u}\in\mathbb{R}^{N}$ is tensorized via a fixed
index-encoding map (Section~\ref{sec:preliminaries}) by reshaping into an
order-$n$ tensor with local dimension $d$, such that $N=d^{n}$ (binary/QTT
corresponds to $d=2$). The tensor is encoded as an MPS using a prescribed
tensorization architecture $\mathcal{A}$ (e.g., sequential, bit-reversed, or
locality-preserving layouts). This mapping defines a correspondence between
grid indices and MPS sites and is chosen to preserve spatial locality to the
extent possible. Encoded states are optionally brought into mixed-canonical form
to improve numerical conditioning.

\paragraph{Operator Representation}
The one-step update in tensorized coordinates is represented as an MPO (TT
operator) with cores
$W_i \in \mathbb{R}^{R_{i-1}\times d\times d\times R_i}$ and bond dimensions
$R_i \le R_{\max}$.
The number of MPO parameters scales as
$\sum_{i=1}^{n} R_{i-1} d^{2} R_i = \mathcal{O}(n\,R_{\max}^{2}\,d^2)$ under
uniform rank bounds, enabling compact representations of high-dimensional
operators. In practice, such MPOs may arise from a chosen PDE discretization and
time-integration scheme (e.g., finite-difference operators assembled in MPO form
and combined into a one-step map) \cite{KazeevKhoromskij2012QTTLaplace,Peddinti2024CFD}. In our experiments, derivative MPOs are constructed from finite-difference stencils; their MPO ranks are therefore fixed by the operator construction. 

\paragraph{MPO--MPS Application}
Given an input MPS $\Psi_k=\mathrm{encode}(\mathbf{u}_k)$, the (pre-compression)
next tensorized state is obtained by contracting each MPO core with the
corresponding MPS core:
\[
\widetilde{\Psi}_{k+1} = \mathcal{T}(\Psi_k).
\]
This contraction sums over physical indices and combines virtual indices,
yielding an MPS whose intermediate bond dimensions typically increase due to the
combined MPO and MPS ranks.

\paragraph{Canonicalization and Compression}
To control rank growth, the predicted MPS is canonicalized and compressed using
an SVD-based sweep \cite{Schollwock2011,Orus2014}. At each bond, singular values
below a threshold $\varepsilon_{\mathrm{SVD}}$ are discarded and the MPS bond
dimension is capped at $\chi_{\max}$. We denote the resulting projection by
$\Pi_{\chi_{\max},\varepsilon_{\mathrm{SVD}}}(\cdot)$, where
$\Pi_{\chi_{\max},\varepsilon_{\mathrm{SVD}}}$ comprises a canonicalization sweep
followed by SVD truncation at each bond. 
This procedure improves numerical conditioning and prevents uncontrolled rank
growth, which is particularly important in transport-dominated or nonlinear
regimes.

\paragraph{Rollout and Reconstruction}
Given an initial condition $\mathbf{u}_0$, we iterate the projected one-step map
to generate a rollout in tensorized coordinates:
$\Psi_{k+1}^{\mathrm{QTN}}=\Pi_{\chi_{\max},\varepsilon_{\mathrm{SVD}}}(\mathcal{T}(\Psi_k^{\mathrm{QTN}}))$,
with $\Psi_0^{\mathrm{QTN}}=\mathrm{encode}(\mathbf{u}_0)$.
Dense reconstruction $\widehat{\mathbf{u}}_k=\mathrm{decode}(\Psi_k^{\mathrm{QTN}})$
is performed only for visualization and for comparison with a reference solver.

\begin{algorithm}[t]\small
\caption{Quantum-Inspired Tensor-Network (QTN) Time Stepping (MPO--MPS with Truncation)}
\label{alg:qtn}
\begin{algorithmic}[1]
\REQUIRE Initial state $\mathbf{u}_0$,
tensorization architecture $\mathcal{A}$,
one-step MPO $\mathcal{T}$ (for step size $\Delta t$, bonds $\le R_{\max}$),
MPS bond cap $\chi_{\max}$,
SVD threshold $\varepsilon_{\mathrm{SVD}}$,
number of steps $K$
\ENSURE Rollout $\{\Psi_k^{\mathrm{QTN}}\}_{k=0}^{K}$ (and optionally $\{\widehat{\mathbf{u}}_k\}$)

\STATE Encode initial state $\Psi_0^{\mathrm{QTN}} \gets \mathrm{encode}(\mathbf{u}_0;\mathcal{A})$.
\FOR{$k=0$ to $K-1$}
    \STATE $\widetilde{\Psi}_{k+1} \gets \mathcal{T}(\Psi_k^{\mathrm{QTN}})$
    \COMMENT{MPO--MPS contraction}
    \STATE $\Psi_{k+1}^{\mathrm{QTN}} \gets \Pi_{\chi_{\max},\varepsilon_{\mathrm{SVD}}}(\widetilde{\Psi}_{k+1})$
    \COMMENT{canonicalize \& truncate}
    \STATE 
    $\widehat{\mathbf{u}}_{k+1} \gets \mathrm{decode}(\Psi_{k+1}^{\mathrm{QTN}})$
\ENDFOR
\STATE \textbf{return} $\{\Psi_k^{\mathrm{QTN}}\}_{k=0}^{K}$ (and optionally $\{\widehat{\mathbf{u}}_k\}$)
\end{algorithmic}
\end{algorithm}

\begin{remark}
The QTN framework serves as a compressed time-stepping baseline for PDE flow maps.
Unlike tensor-network learning approaches that
identify a time-stepping operator from data, here
the one-step MPO $\mathcal{T}$ represents a structured update in tensorized
coordinates (e.g., arising from a chosen discretization and time integrator),
while $\Pi_{\chi_{\max},\varepsilon_{\mathrm{SVD}}}$ controls complexity by truncating
rank growth during rollout. This setting enables a focused assessment of the
capabilities and limitations of low-rank tensor networks for scalable PDE time
stepping and provides a baseline for future physics-informed or data-driven
extensions.
\end{remark}

\subsection{Theoretical Properties}
\label{sec:theory}
The tensor-network formulation employed in this work is grounded in standard
linear-algebraic properties of matrix product states and operators. These
properties provide structural insight into the proposed QTN models and clarify
both their expressive capabilities and intrinsic limitations for PDE time
stepping and operator approximation. In particular, they address: (i)
representability of tensorized discretized states in MPS/TT form, (ii) local
optimality of SVD-based truncation used during compression, and (iii)
accumulation of one-step approximation errors under repeated time stepping.

\begin{proposition}[Exact MPS/TT representation
]
\label{prop:mps-exact}
Let $u \in \mathbb{R}^{d^n}$ be any discretized state reshaped into an order-$n$
tensor $\mathcal{U} \in \mathbb{R}^{d \times \cdots \times d}$. Then there exists
an exact (i.e., without truncation) open-boundary MPS/TT representation of $\mathcal{U}$ with bond
dimensions satisfying
\[
D_i \le d^{\min(i,n-i)}, \qquad i=1,\dots,n-1.
\]
In particular, any such tensor admits an exact MPS representation with maximal
bond dimension at most $d^{\lfloor n/2 \rfloor}$, attained at the cut between
sites $\lfloor n/2 \rfloor$ and $\lfloor n/2 \rfloor+1$.
\end{proposition}

\begin{proof}
The construction follows from successive applications of the singular value
decomposition (SVD). First reshape $\mathcal{U}$ into a matrix of size
$d \times d^{n-1}$ by grouping indices as $(s_1)$ and $(s_2,\dots,s_n)$, and
perform an SVD:
\[
\mathcal{U}_{s_1,(s_2 \dots s_n)}
= \sum_{\alpha_1}
L_{s_1 \alpha_1}\, \Sigma_{\alpha_1}\,
R^\top_{\alpha_1,(s_2 \dots s_n)}.
\]
Define the first MPS core by $A_1^{(s_1)}[1,\alpha_1] := L_{s_1\alpha_1}$ and
absorb $\Sigma$ into the remaining tensor
$\widetilde{\mathcal{U}}^{(2)}_{\alpha_1,(s_2 \dots s_n)}
:= \Sigma_{\alpha_1} R^\top_{\alpha_1,(s_2 \dots s_n)}$.
At step $i$, reshape the remaining tensor into the unfolding that groups the
first $i$ physical indices and the remaining $n-i$ indices, yielding a matrix
of size $d^i \times d^{n-i}$, and perform an SVD. The rank of this matrix is at
most
\[
\min(d^i, d^{n-i}) = d^{\min(i,n-i)},
\]
which bounds the bond dimension $D_i$. The quantity $\min(i,n-i)$ is maximized
at $i=\lfloor n/2 \rfloor$, giving $\max_i D_i \le d^{\lfloor n/2 \rfloor}$.
After $n-1$ steps, the full tensor $\mathcal{U}$ is exactly represented as an
MPS with the stated bond-dimension bounds.
\end{proof}

\begin{proposition}[Optimality of SVD-based compression]
\label{prop:svd-optimal}
Let $X \in \mathbb{R}^{m \times n}$ have singular value decomposition
$X = U \Sigma V^\top$, and fix an integer $r$ with $0 \le r \le p:=\min(m,n)$.
Let $X_r = U_r \Sigma_r V_r^\top$ be the rank-$r$ truncation obtained by
retaining the $r$ largest singular values. Then $X_r$ is a best rank-$r$
approximation of $X$ in the Frobenius norm, i.e.,
\[
\|X - X_r\|_F
= \min_{\operatorname{rank}(Y) \le r} \|X - Y\|_F.
\]
In particular, the SVD-based truncation used in MPS/MPO compression is
\emph{locally} Frobenius-norm optimal for each matrix unfolding at the moment of
truncation (though sequential truncations need not yield a globally optimal
tensor-network approximation).
\end{proposition}

\begin{proof}
This result is a direct consequence of the Eckart--Young--Mirsky theorem.
Let the singular values of $X$ be
$\sigma_1 \ge \sigma_2 \ge \dots \ge \sigma_p \ge 0$, where $p = \min(m,n)$.
Then the truncated SVD satisfies
$\|X - X_r\|_F^2 = \sum_{i=r+1}^p \sigma_i^2$,
and for any matrix $Y$ with $\operatorname{rank}(Y) \le r$,
$\|X - Y\|_F^2 \ge \sum_{i=r+1}^p \sigma_i^2$.
Therefore, $X_r$ minimizes the Frobenius norm error among all rank-$r$
approximations of $X$.
\end{proof}

\begin{proposition}[Multi-step error propagation]
\label{prop:error-propagation}
Let $\mathcal{F}_{\Delta t}: \mathbb{R}^N \to \mathbb{R}^N$ denote the exact
discrete-time flow map. Define the implemented QTN one-step predictor in
physical space by
\[
\widehat{\mathcal{F}}(\mathbf{u})
:= \mathrm{decode}\!\left(
\Pi_{\chi_{\max},\varepsilon_{\mathrm{SVD}}}\big(
\mathcal{T}(\mathrm{encode}(\mathbf{u}))
\big)\right),
\]
where $\mathcal{T}$ is an MPO/TT-operator in tensorized coordinates and
$\Pi_{\chi_{\max},\varepsilon_{\mathrm{SVD}}}$ denotes the
canonicalization+truncation projection.

Assume $\mathcal{F}_{\Delta t}$ is Lipschitz on a set $\mathcal{S}\subset\mathbb{R}^N$
containing the exact trajectory $\{\mathbf{u}_k\}_{k\ge 0}$ and the predicted
trajectory $\{\mathbf{u}^{\mathrm{QTN}}_k\}_{k\ge 0}$, with constant $L\ge 0$, that is
\[
\|\mathcal{F}_{\Delta t}(\mathbf{u}) - \mathcal{F}_{\Delta t}(\mathbf{v})\|_2
\le L \|\mathbf{u} - \mathbf{v}\|_2,
\qquad \forall\,\mathbf{u},\mathbf{v}\in\mathcal{S}.
\]
Define the uniform one-step approximation error on $\mathcal{S}$ as
\[
e := \sup_{\mathbf{u} \in \mathcal{S}}
\|\widehat{\mathcal{F}}(\mathbf{u}) - \mathcal{F}_{\Delta t}(\mathbf{u})\|_2.
\]
Let the exact and QTN rollouts be generated by
\[
\mathbf{u}_{k+1} = \mathcal{F}_{\Delta t}(\mathbf{u}_k),
\qquad
\mathbf{u}_{k+1}^{\mathrm{QTN}} = \widehat{\mathcal{F}}(\mathbf{u}_k^{\mathrm{QTN}}),
\qquad
\mathbf{u}_0^{\mathrm{QTN}}=\mathbf{u}_0.
\]
Then, for any $m \ge 1$,
\[
\|\mathbf{u}_{m}^{\mathrm{QTN}} - \mathbf{u}_{m}\|_2
\le
\sum_{j=0}^{m-1} L^{j}\, e.
\]
\end{proposition}

\begin{proof}
Define $\delta_k := \|\mathbf{u}_k^{\mathrm{QTN}} - \mathbf{u}_k\|_2$.
For any $k \ge 0$, add and subtract $\mathcal{F}_{\Delta t}(\mathbf{u}_k^{\mathrm{QTN}})$:
\begin{align*}
\delta_{k+1}
&= \|\widehat{\mathcal{F}}(\mathbf{u}_k^{\mathrm{QTN}})
      - \mathcal{F}_{\Delta t}(\mathbf{u}_k)\|_2 \\
&\le
\|\widehat{\mathcal{F}}(\mathbf{u}_k^{\mathrm{QTN}})
  - \mathcal{F}_{\Delta t}(\mathbf{u}_k^{\mathrm{QTN}})\|_2
+
\|\mathcal{F}_{\Delta t}(\mathbf{u}_k^{\mathrm{QTN}})
  - \mathcal{F}_{\Delta t}(\mathbf{u}_k)\|_2 \\
&\le e + L\,\delta_k,
\end{align*}
where we used the definition of $e$ and Lipschitz continuity on $\mathcal{S}$.
Since $\delta_0=0$, unrolling the recursion yields
$\delta_m \le \sum_{j=0}^{m-1} L^{j} e$.
\end{proof}

\begin{remark}[Restart-time bound and interpretation]
\label{rem:restart-bound}
More generally, if the QTN rollout is initialized at some time $k$ with
possibly mismatched state $\mathbf{u}_k^{\mathrm{QTN}} \neq \mathbf{u}_k$, then
for any $m\ge 1$,
\[
\|\mathbf{u}_{k+m}^{\mathrm{QTN}} - \mathbf{u}_{k+m}\|_2
\le
L^{m}\|\mathbf{u}_k^{\mathrm{QTN}} - \mathbf{u}_k\|_2
+
\sum_{j=0}^{m-1} L^{j}\, e.
\]
Thus, even small one-step errors can accumulate under repeated time stepping, and the amplification is governed by the Lipschitz constant $L$ of the exact flow map on the relevant trajectory set. In particular, contractive dynamics ($L < 1$) keep QTN errors uniformly bounded, neutral dynamics ($L \approx 1$) lead to at most linear growth with the prediction horizon, and expansive dynamics ($L > 1$) can amplify small discrepancies rapidly.

\end{remark}

\section{Numerical Experiments}
\label{sec:experiments}

This section evaluates the proposed QTN framework for approximating discrete-time
PDE flow maps via compressed time stepping. In all experiments, we generate
reference solution trajectories using a classical time-integration scheme and
compare the QTN rollout against the reference rollout. At inference, we roll out
the projected one-step predictor
\[
\widehat{\mathcal{F}}(\mathbf{u})
= \mathrm{decode}\!\left(
\Pi_{\chi_{\max},\varepsilon_{\mathrm{SVD}}}\big(
\mathcal{T}(\mathrm{encode}(\mathbf{u}))
\big)\right)
\]
for multiple steps and compare with the reference. We report qualitative
comparisons via snapshots and signed difference plots, and quantify error growth
using relative $\ell_2$ norms and prediction-horizon curves.

Common tensorization and truncation settings:
Unless stated otherwise, we use base-$d$ tensorization with $N=d^n$ (binary/QTT
corresponds to $d=2$), a locality-preserving architecture $\mathcal{A}$, an MPS
bond cap $\chi_{\max}$, an MPO bond cap $R_{\max}$, and SVD threshold
$\varepsilon_{\mathrm{SVD}}$. After each MPO application, we apply
$\Pi_{\chi_{\max},\varepsilon_{\mathrm{SVD}}}$ to control rank growth, as
described in Section~\ref{sec:methods}.

\paragraph{1D Linear Advection--Diffusion}
We consider $u_t + c u_x = \nu u_{xx}$ on $[0,1]$ with homogeneous Dirichlet
boundary conditions and initial condition $u_0(x)=\exp(-100(x-0.3)^2)$, using
$c=0.5$, $\nu=0.01$, $T=0.5$, and $N=2^9=512$ grid points. 
We use an MPS bond cap $\chi_{\max}=60$ and truncation tolerance $\varepsilon_{\mathrm{SVD}}=10^{-12}$, applied after each Euler step (after enforcing Dirichlet boundary conditions); the Dirichlet mask MPS is compressed with bond cap $30$.
Reference trajectories
are generated by a method-of-lines discretization (second-order centered
differences) and \texttt{RK45}. We implement an MPO-based one-step update
$\mathcal{T}$ in tensorized coordinates with binary tensorization ($d=2$, $n=9$),
applying $\Pi_{\chi_{\max},\varepsilon_{\mathrm{SVD}}}$ after each MPO
application. We then roll out the projected one-step predictor and compare with
the RK45 rollout. 
Figure~\ref{fig:advdiff_1d_comparison}(a--c) shows the QTN
rollout, reference, and signed difference $u_{\mathrm{RK45}}-u_{\mathrm{QTN}}$
(time-aligned). 

Figure~\ref{fig:advdiff_1d_comparison}(d) reports the restart-averaged relative $\ell_2$ error as a function of prediction horizon $m$. Specifically, for each restart time $k$, we initialize the QTN model at $u_k$ and measure
$\frac{\|u^{\mathrm{QTN}}_{k+m} - u_{k+m}\|_2}{\|u_{k+m}\|_2}$,
then average over all valid restart times $k$.
The figure shows small and weakly increasing error, consistent with stable, diffusion-dominated dynamics.

\begin{figure}[t]
    \centering
    \includegraphics[width=\linewidth]{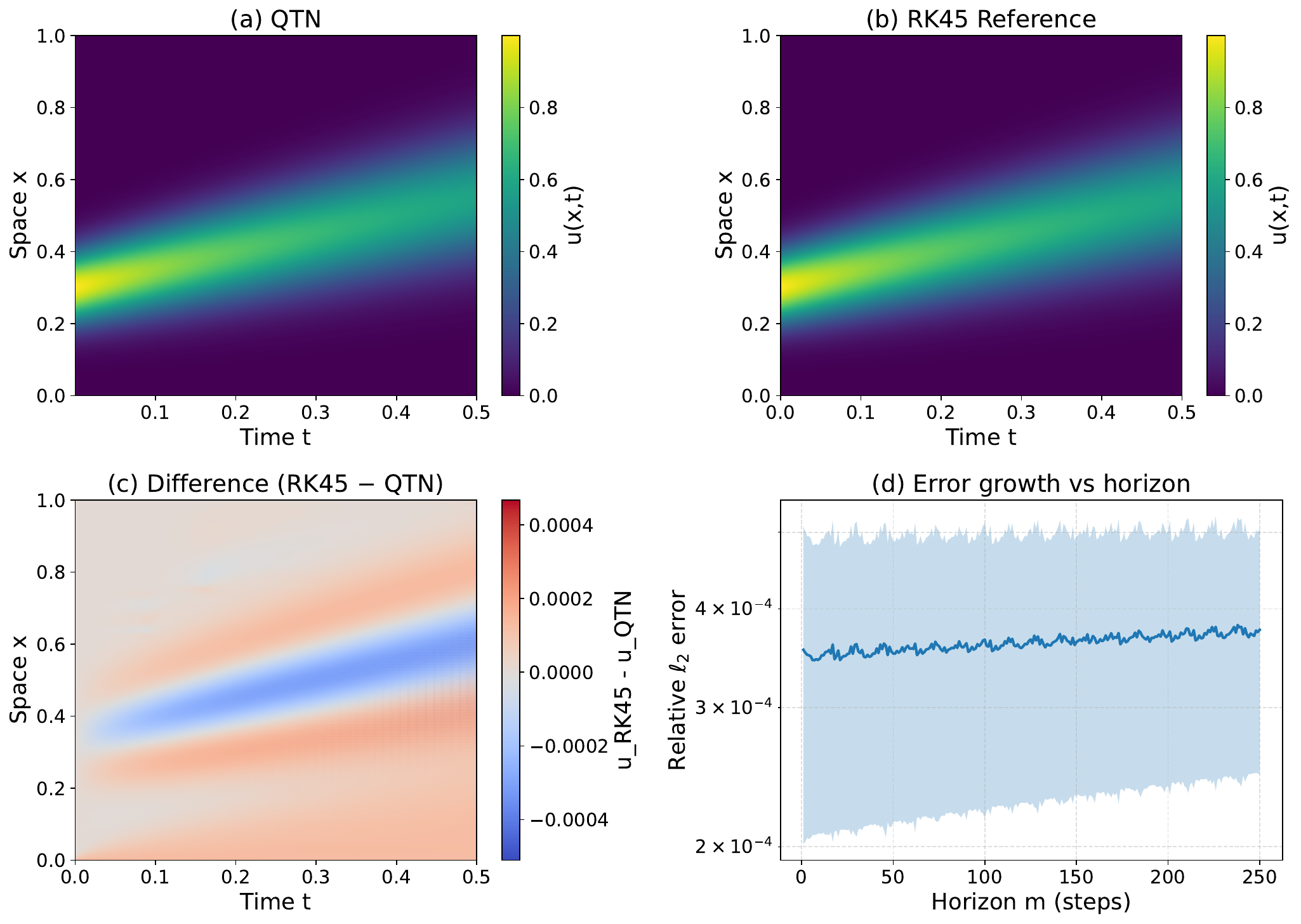}
   \caption{\small
1D advection--diffusion: (a) QTN rollout from iterating the one-step MPO predictor
with truncation; (b) RK45 reference; (c) signed difference
$u_{\mathrm{RK45}}-u_{\mathrm{QTN}}$ (time-aligned); (d) restart-averaged relative
$\ell_2$ error vs.\ prediction horizon $m$.
}
    \label{fig:advdiff_1d_comparison}
\end{figure}

\paragraph{2D Linear Advection--Diffusion}
We consider the two-dimensional advection--diffusion equation
$u_t + c_x u_x + c_y u_y = \nu(u_{xx}+u_{yy})$ on $[0,1]^2$ with homogeneous
Dirichlet boundary conditions. 
We cap the state bond dimension at $\chi_{\max}=80$ and apply truncation with tolerance $\varepsilon_{\mathrm{SVD}}=10^{-12}$ after each Euler step (after applying the Dirichlet mask); the boundary mask MPS is compressed with bond cap $40$.
The initial condition is a Gaussian pulse
centered at $(0.3,0.4)$ with $(c_x,c_y)=(0.5,0.2)$, $\nu=0.01$, final time $T=1$,
and a $2^6\times 2^6$ grid (total $N=2^{12}$). Reference trajectories are
generated by a method-of-lines discretization (centered finite differences) and
\texttt{RK45}.
We flatten the 2D grid index and apply a locality-preserving tensorization
architecture $\mathcal{A}$ to encode states as MPS, then implement an MPO-based
one-step update $\mathcal{T}$ with truncation via
$\Pi_{\chi_{\max},\varepsilon_{\mathrm{SVD}}}$ after each MPO application. We
evaluate multi-step rollouts by iterating the projected one-step predictor and
comparing to RK45 in Figure~\ref{fig:advdiff_2d_comparison}. It can be seen that the errors remain
spatially smooth over short to moderate horizons, indicating that low-rank MPO
updates capture the dominant transport--diffusion dynamics in this regime.
\begin{figure}[t]
  \centering
  \includegraphics[width=\linewidth]{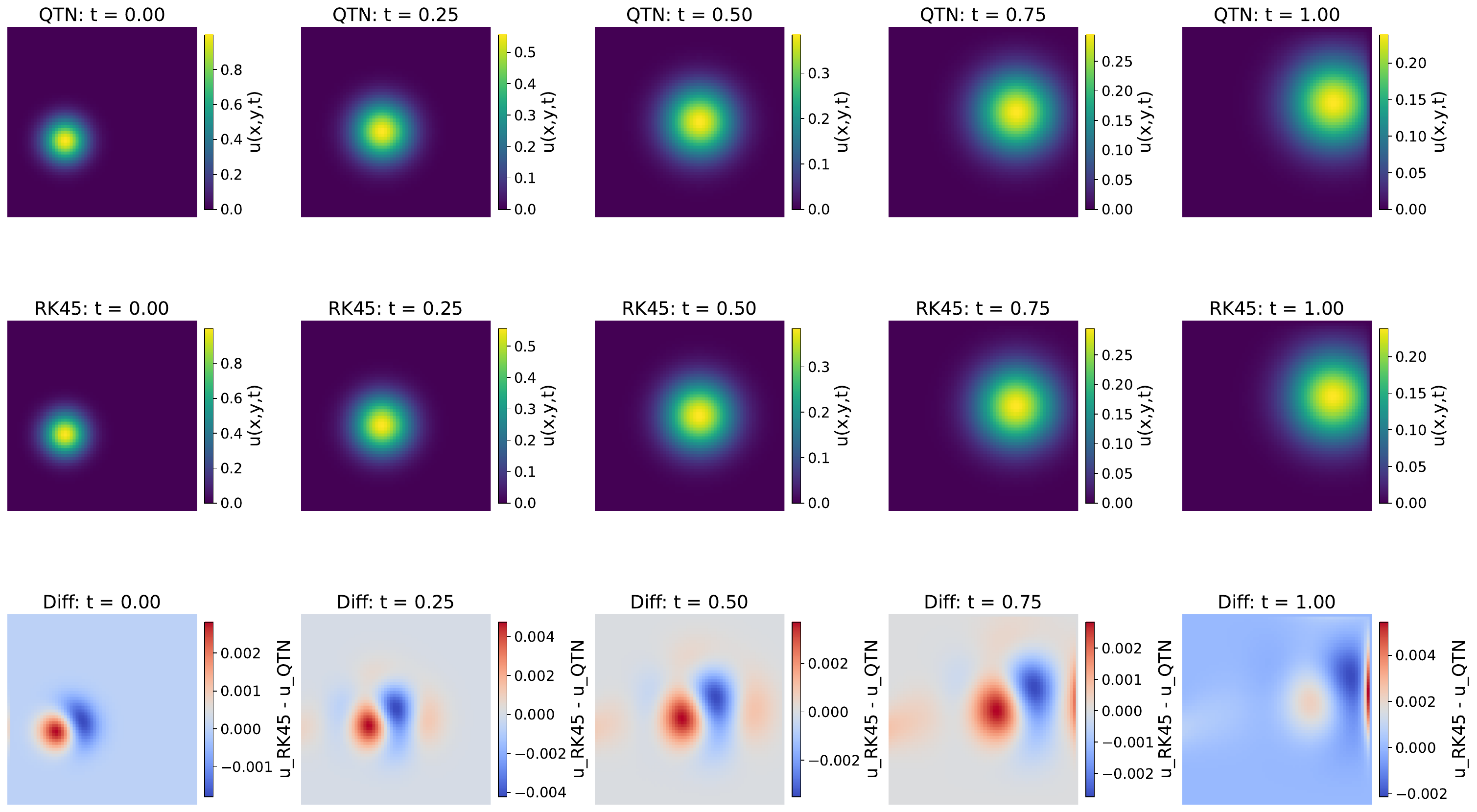}
  \caption{\small
  2D advection--diffusion: comparison of QTN rollout and RK45 reference at
  $t=\{0,\,0.25,\,0.50,\,0.75,\,1.00\}$.
  Top row: QTN rollout obtained by iterating the one-step predictor with
  truncation after each step (states stored in compressed MPS form).
  Middle row: RK45 reference solution at matching times.
  Bottom row: signed difference $u_{\mathrm{RK45}}-u_{\mathrm{QTN}}$.
  }
  \label{fig:advdiff_2d_comparison}
\end{figure}

\paragraph{1D Viscous Burgers Equation}
To assess performance beyond linear dynamics, we consider the viscous Burgers
equation $u_t + u u_x = \nu u_{xx}$ on $[0,1]$ with periodic boundary conditions
and initial condition $u_0(x)=1+0.5\sin(2\pi x)$, viscosity $\nu=0.01$, and final
time $T=0.5$. The spatial grid uses $N=2^9$ points. We use an MPS bond cap
$\chi_{\max}=60$ and truncation tolerance $\varepsilon_{\mathrm{SVD}}=10^{-12}$
applied after each Euler step; periodic boundary conditions are encoded directly
in the finite-difference derivative MPOs.
A reference solution is computed via method-of-lines discretization with periodic
finite differences and \texttt{RK45} time integration. We implement an MPO-based
one-step update $\mathcal{T}$ in tensorized coordinates with truncation after each
application. Figure~\ref{fig:burgers_1d_comparison} compares QTN rollouts with the
RK45 reference, shows the signed difference, and reports restart-averaged relative
$\ell_2$ error versus prediction horizon.
The signed difference field in Figure~\ref{fig:burgers_1d_comparison}(c) remains spatially smooth and aligned with the transported wave structure, indicating that discrepancies arise primarily from gradual phase and amplitude mismatch rather than numerical instability or spurious oscillations.
Figure~\ref{fig:burgers_1d_comparison}(d) reports the restart-averaged relative $\ell_2$ error versus prediction horizon $m$, showing more fluctuating error than in the linear advection–diffusion case, reflecting the increased sensitivity of the nonlinear dynamics, yet remaining stable over the tested horizon.

\begin{figure}[t]
  \centering
  \includegraphics[width=\linewidth]{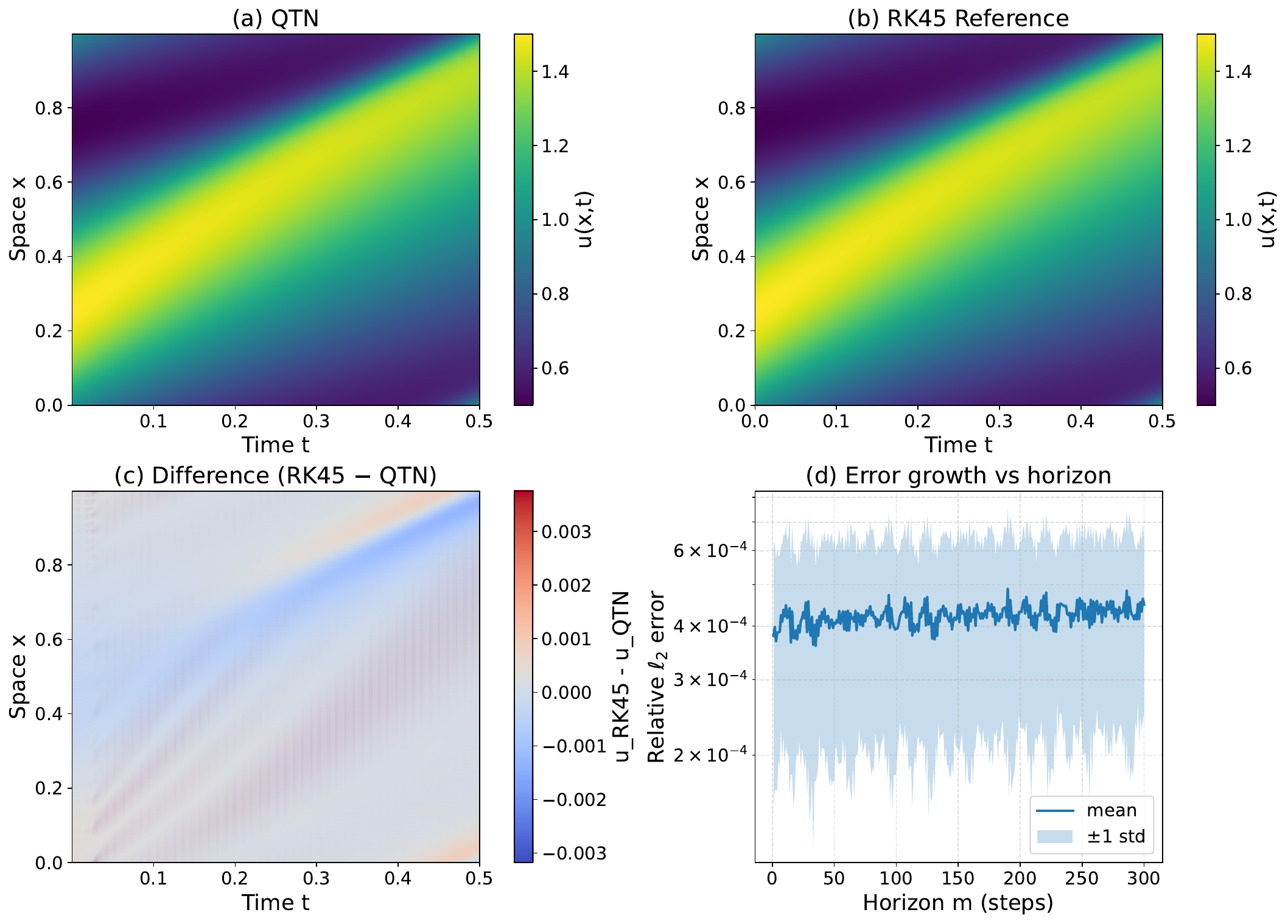}
  \caption{\small
  1D viscous Burgers (periodic): (a) QTN rollout, (b) RK45 reference,
  (c) signed difference $u_{\mathrm{RK45}}-u_{\mathrm{QTN}}$ (time-aligned),
  and (d) restart-averaged relative $\ell_2$ error versus prediction horizon
  $m$ (mean $\pm 1$ std).}
  \label{fig:burgers_1d_comparison}
\end{figure}

\paragraph{2D Burgers Equation (Periodic)} We consider the two-dimensional viscous Burgers equation $u_t + u\,u_x + u\,u_y = \nu (u_{xx} + u_{yy})$ on
$[0,1]^2$ with
periodic boundary conditions and viscosity $\nu=0.01$, using a $64\times 64$
grid (total $N=4096=2^{12}$) and final time $T=1$. 
We use an MPS bond cap $\chi_{\max}=120$ and truncation tolerance $\varepsilon_{\mathrm{SVD}}=10^{-12}$ applied after each Euler step; periodic boundary conditions are encoded directly in the finite-difference derivative MPOs.
Reference trajectories are
generated by a method-of-lines discretization with centered finite differences
and \texttt{RK45} time integration. We encode 2D states in MPS form using a
locality-preserving tensorization architecture and evaluate multi-step rollouts
by iterating the one-step predictor with truncation after each step.
Figure~\ref{fig:burgers_2d_comparison}(a--c) compares the QTN rollout with
the RK45 reference and shows the signed difference $u_{\mathrm{RK45}}-u_{\mathrm{QTN}}$
at representative times. The difference fields remain spatially smooth over the
shown horizon, while localized discrepancies gradually develop as nonlinear
interactions accumulate.

\begin{figure}[t]
  \centering
  \includegraphics[width=\linewidth]{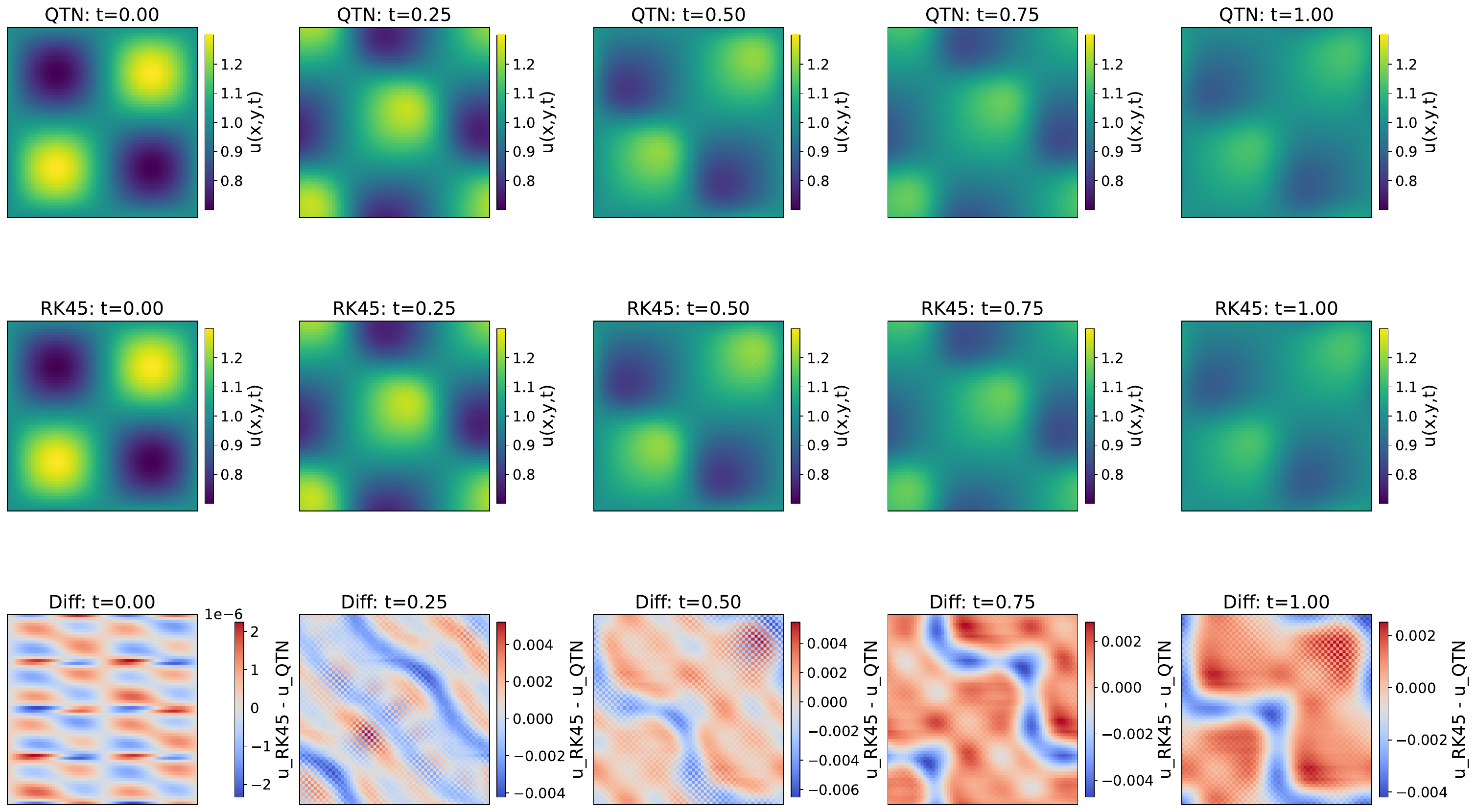}
  \caption{\small
  2D viscous Burgers: comparison of QTN rollout and RK45 reference at
  $t=\{0,\,0.25,\,0.50,\,0.75,\,1.00\}$.
  Top row: QTN rollout obtained by iterating the one-step predictor with
  truncation after each step.
  Middle row: RK45 reference snapshots at matching times.
  Bottom row: signed difference $u_{\mathrm{RK45}}-u_{\mathrm{QTN}}$.}
  \label{fig:burgers_2d_comparison}
\end{figure}

\section{Conclusion}
We investigated quantum-inspired tensor networks as a fully classical framework
for approximating discrete-time flow maps of hydrodynamic PDEs via compressed
time stepping. The proposed approach encodes tensorized PDE states in MPS/TT
form and represents one-step evolution updates as low-rank MPO/TT-operators, with
rank growth controlled by canonicalization and SVD-based truncation. This yields
an interpretable tensor-network baseline that avoids dense representations and
does not impose explicit physical constraints (e.g., conservation, reversibility,
or unitarity), making it applicable to dissipative dynamics.

Numerical experiments on one- and two-dimensional linear advection--diffusion
and viscous Burgers equations show that low-rank MPO updates can accurately
capture \emph{short-horizon} dynamics in smooth, diffusion-dominated regimes,
while multi-step rollouts exhibit increasing error consistent with accumulated
one-step mismatch and compression effects. For nonlinear PDEs, the projected
low-rank update should be interpreted as an \emph{approximation} of the one-step
map on the states encountered along the rollout, which explains both its
effectiveness for local prediction and its limitations for long-horizon
extrapolation.

This work establishes quantum-inspired tensor networks as a principled,
structured, and scalable baseline for PDE flow-map approximation. Promising
future directions include incorporating physics-based inductive biases (e.g.,
conservation or stability constraints), stability-aware and multi-step objectives,
adaptive rank/truncation strategies, and extensions toward data-driven operator
identification and hardware-compatible or quantum-motivated parameterizations.

\bibliographystyle{IEEEtran}
\bibliography{IEEEabrv,qccl}

\end{document}